\documentclass{article}

\usepackage[T2A]{fontenc}
\usepackage[utf8]{inputenc}
\usepackage[russian,english]{babel}
\usepackage{amsfonts}
\usepackage{amsthm}
\usepackage{amsmath}
\usepackage{amssymb}
\usepackage{cancel}
\usepackage{tabu}
\usepackage{array}
\usepackage{tikz}
\usepackage{makecell}
\usepackage{xcolor}
\usepackage{graphicx}
\usepackage[final]{pdfpages}
\usepackage{bbm}

\usepackage{titlesec}
\usepackage{indentfirst}
\usepackage{enumitem}
\usepackage{caption}
\usepackage{refcount}
\usepackage{cite}
\usepackage[colorlinks=true, linkcolor=blue, citecolor=blue, urlcolor=blue, hypertexnames=false]{hyperref}
\usepackage{stmaryrd,scalerel}
\usepackage{titling}
\usepackage{etoc}

\titleformat{\section}[block]          
  {\centering\large\bfseries}          
  {\S\,\thesection.}                   
  {1em}                                
  {}                                   

\makeatletter
\renewenvironment{proof}[1][\proofname]{\par
  \pushQED{\qed}%
  \normalfont \topsep6\p@\@plus6\p@\relax
  \trivlist
  \item[\hspace*{\parindent}\hskip\labelsep \bfseries #1\@addpunct{.}]%
  \ignorespaces
}{%
  \popQED\endtrivlist\@endpefalse
}
\makeatother

\graphicspath{ {./images/} }

\renewcommand{\ge}{\geqslant}

\renewcommand{\le}{\leqslant}

\newtheoremstyle{style2}
  {\topsep}
  {\topsep}
  {\itshape}
  {\parindent}
  {\bfseries}
  {.}
  {5pt plus 1pt minus 1pt}
  {}

\title{Граница Мартина $r$-дифференциальной версии графа Юнга --- Фибоначчи}

\author{В. Ю. Евтушевский}

\begin{document}

\selectlanguage{russian}

\maketitle

\etocsettagdepth{ru}{all}
\etocsettagdepth{en}{none}
\tableofcontents

\theoremstyle{style2}
\newtheorem{Lemma}{Лемма}

\theoremstyle{style2}
\newtheorem{Alg}{Алгоритм}

\theoremstyle{style2}
\newtheorem{Col}{Следствие}

\theoremstyle{style2}
\newtheorem{theorem}{Теорема}

\theoremstyle{style2}
\newtheorem{Def}{Определение}

\theoremstyle{style2}
\newtheorem{Prop}{Предложение}

\theoremstyle{style2}
\newtheorem{Problem}{Задача}

\theoremstyle{style2}
\newtheorem{Zam}{Замечание}

\theoremstyle{style2}
\newtheorem{Oboz}{Обозначение}

\theoremstyle{style2}
\newtheorem{Ex}{Пример}

\theoremstyle{style2}
\newtheorem{Nab}{Наблюдение}

\begin{abstract}
Для $r$-дифференциальной версии графа Юнга --- Фибоначчи получены формулы для числа путей между вершинами; описана граница Мартина пространства путей; и доказывается эргодичность соответствующих мер.
\end{abstract}

\etocdepthtag.toc{ru}
\section{Введение}

При целом положительном $r$ рассмотрим слова над алфавитом 
размера $r+1$, состоящем из двойки и $r$ разноцветных единиц,
имеющие данную сумму цифр $n.$ Как известно, их количество есть $r$-число Фибоначчи $F_{r,n+1}$ ($F_{r,0}=0,$ $F_{r,1}=1,$ $F_{r,k+2}=r\cdot F_{r,k+1}+F_{r,k}$), и это самая распространённая комбинаторная интерпретация $r$-чисел Фибоначчи. Также можно думать о разбиениях полосы $2\times n$ на домино  $1\times 2$ и $2\times 1$, 
где вертикальные домино  могут быть покрашены в $r$ цветов
(двойки соответствуют парам горизонтальных домино, а единицы ---вертикальным домино.)

При каждом целом положительном $r$ введём на соответствующем множестве слов частичный порядок: будем говорить, что слово $v$ \emph{предшествует} слову $w$, если после удаления самого длинного общего суффикса в слове $w$ остаётся не меньше двоек, чем в слове $v$ остаётся символов. 

Это действительно частичные порядки, более того, соответствующие частично упорядоченные множества являются модулярными решётками. При $r=1$ такая решётка известна как решётка Юнга --- Фибоначчи, а при произвольном $r$ --- как $r$-дифференциальная версия решётки Юнга --- Фибоначчи.  

\begin{figure}[htbp]
    \centering
    \includegraphics[width=12cm, height=10cm]{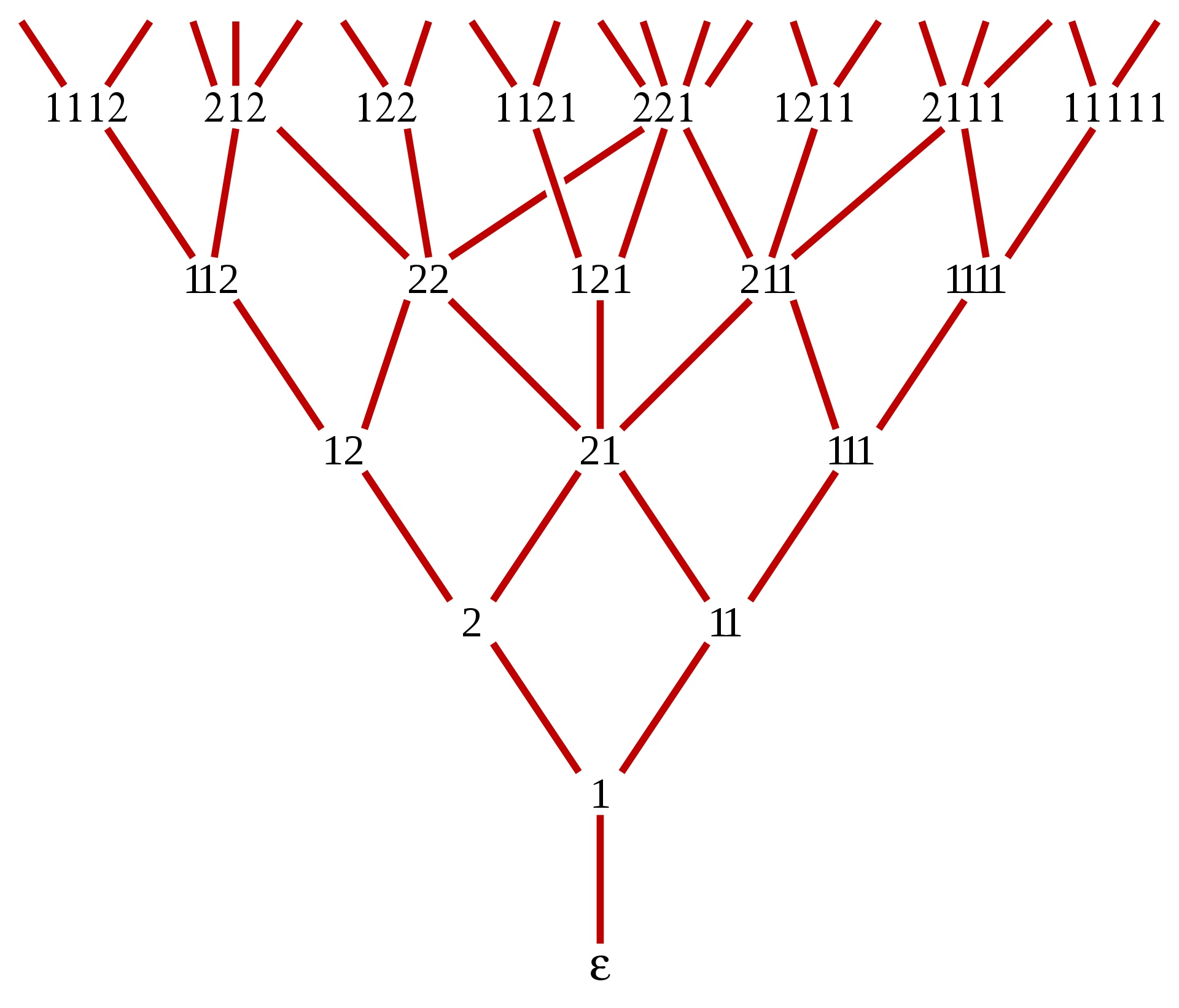}
    \caption{Граф Юнга --- Фибоначчи}
    \label{yf}
\end{figure}

Диаграммы Хассе этих решёток называются графом Юнга --- Фибоначчи (он изображён на рис. \ref{yf}) и $r$-дифференциальной версией графа Юнга --- Фибоначчи соответственно. Это градуированные графы, которые мы представляем растущими снизу вверх, начиная с пустого слова. Градуировкой служит функция суммы цифр (считаем, что единица 
любого цвета вносит вклад $1$ в сумму цифр). Опишем явно, как устроены ориентированные рёбра. Рёбра <<вверх>> из данного слова $v$ ведут в слова, получаемые из $v$ одной из двух операций:
\begin{enumerate}
    \item заменить самую левую единицу (любого цвета) на двойку;
    \item вставить единицу (любого цвета) в такое место,
    что левее неё нет единиц.
\end{enumerate}

В данной работе будем считать эти графы ориентированными и содержащими исключительно рёбра <<вверх>>.

При фиксированном $r$ такой граф помимо модулярности является $r$-дифференциальным, то есть для каждой вершины исходящая степень на $r$ превосходит входящую степень.

Изучение градуированного графа Юнга --- Фибоначчи было инициировано в 1988 г. одновременно и независимо друг от друга такими математиками, как Ричард Стенли \cite{St} и Сергей Владимирович Фомин \cite{Fo}.

Одна из причин интереса к нему в том, что существует всего две корневые локально конечные $1$-дифференциальные модулярные решётки, вторая --- это решётка диаграмм Юнга, имеющая ключевое значение в теории представлений симметрической группы \cite[Theorem~1.3]{Bur}. 

Ключевые вопросы о градуированных графах касаются центральных мер на пространстве (бесконечных) путей в графе. Эта точка зрения последовательно развивалась в работах Анатолия Моисеевича Вершика, к недавнему обзору которого \cite{Ver} и приводимой там литературе мы отсылаем читателя.

Среди центральных мер выделяют те, которые являются пределами мер, индуцированных путями в далёкие вершины --- так называемую \emph{границу Мартина} графа.

Граница пространства путей графа Юнга --- Фибоначчи изучалась в работе Фредерика Гудмана и Сергея Васильевича Керова (2000) \cite{GK}. Они использовали алгебраический формализм Окады \cite{Okada}.

Как следует из самого определения, асимптотический вопрос о границе напрямую связан с перечислительным вопросом о числе путей между двумя вершинами графа. Отметим важную общую работу Фомина \cite{Fo1} о перечислении путей в градуированных графах, в которой приводится ряд общих тождеств и указывается связь, помимо прочего, с обобщением алгоритма Робинсона --- Шенстеда --- Кнута.

В работе Александра Васильевича Гнедина и С.~В.~Керова \cite{KerGned} была доказана эргодичность так называемой меры Планшереля графа Юнга --- Фибоначчи. В работе автора \cite{Evtuh1} (краткая версия которой была опубликована как \cite{Evtuh2}) была предложена явная формула для числа путей между двумя вершинами в графе Юнга --- Фибоначчи, с 
не более чем квадратичным по длине меньшего из слов числом слагаемых. В работе Ивана Алексеевича Бочкова и автора \cite{BE} для всех мер границы Мартина графа Юнга --- Фибоначчи была найдена явная формула для меры множества всех бесконечных путей из корня, проходящих через данную вершину, после чего была доказана эргодичность этих мер.

Целью настоящей работы является перенос вышеизложенных результатов  работ \cite{GK, KerGned, Evtuh1, Evtuh2, BE} на случай произвольного $r.$

В разделе 3 мы предлагаем явные формулы для числа путей между двумя вершинами в $r$-дифференциальной версии графа Юнга --- Фибоначчи с 
не более чем кубическим по длине меньшего из слов числом слагаемых. В разделе 5 мы полностью описываем границу Мартина этого графа. Для каждой меры этой границы явная формула для меры множества всех бесконечных путей из корня, проходящих через данную вершину, следует из доказательства, но не приводится ввиду своей громоздкости. В разделе 6 доказывается эргодичность всех этих мер. Как обычно, для доказательства эргодичности используется \emph{эргодический метод} Вершика.

\section{Основные понятия и обозначения}

\renewcommand{\labelitemii}{$\bullet$} 

Через $\mathbb{N}_0$ и $\mathbb{N}$ обозначаются множества целых неотрицательных и целых положительных (натуральных) чисел, соответственно. 
Для  $m,n\in \mathbb{N}_0:$ $m\leqslant n$ положим
\begin{gather*}
	\overline{m,n}:=\{m,m+1,\ldots,n\},\quad \overline{n}:=\overline{0,n}=\{0,1,\ldots,n\},\quad \overline{m,\infty}:=\{m,m+1,\ldots\}.
\end{gather*}

Рассмотрим алфавит $\mathcal{A}_r$, состоящий из двойки и $r$ разноцветных единиц.
При фиксированном $r\in\mathbb{N}$ пусть ${}^r\mathbb{YF}$ --- это $r$-дифференциальная версия графа Юнга --- Фибоначчи, его вершины --- конечные слова над алфавитом $\mathcal{A}_r$. Далее ${}^r\mathbb{YF}$ как множество --- это множество именно вершин этого графа. Кроме того, граф Юнга --- Фибоначчи ${}^1\mathbb{YF}$ также будем обозначать через $\mathbb{YF}$ и рассматривать его как граф над множеством конечных слов из $1$ и $2$.

Если $r\in\mathbb{N},$ $\{\alpha_i\}_{i=1}^\infty\in\mathcal{A}_r^\infty$ --- \emph{бесконечная} последовательность из двоек и цветных единиц, сопоставим ей <<бесконечно удалённую вершину>> $w=\ldots\alpha_2\alpha_1$ графа ${}^r\mathbb{YF}.$ Множество <<бесконечно удалённых вершин>> графов ${}^r\mathbb{YF}$ и $\mathbb{YF}$ обозначим через ${}^r\mathbb{YF}_\infty$ и $\mathbb{YF}_\infty$ соответственно. Для $w\in{}^r\mathbb{YF}\cup{}^r\mathbb{YF}_\infty,$ $v \in {}^r\mathbb{YF}$ определено слово $wv$ --- конкатенация $w$ и $v$. Также для $v\in{}^r\mathbb{YF}$ и $n\in\mathbb{N}_0$ естественным образом определяется слово $v^n$.

Для $v \in {}^r\mathbb{YF}$ сумму цифр слова $v$ обозначим через $|v|$. Число $|v|$ --- \emph{ранг} вершины $v$ в графе ${}^r\mathbb{YF}$ (считаем, что пустое слово имеет ранг $0$ и переход по ребру наверх увеличивает ранг на $1$). Для $w \in {}^r\mathbb{YF}_\infty$ полагаем $|w|=\infty.$

Для целого неотрицательного $n$ множество вершин графа ${}^r\mathbb{YF}$ ранга $n$ обозначим через
\begin{gather*}
	{}^r\mathbb{YF}_n:=\{v\in{}^r\mathbb{YF}:\;|v|=n\}.
\end{gather*}	
В данной работе будем считать графы ${}^r\mathbb{YF}$ ориентированными с рёбрами из вершин меньшего ранга в вершины большего ранга.

Кроме того, каждой вершине $w\in{}^r\mathbb{YF}\cup{}^r\mathbb{YF}_\infty$ сопоставим вершину $\overline{w}\in\mathbb{YF}\cup\mathbb{YF}_\infty,$ которая получается забыванием цвета единиц.

Для $v,w \in {}^r\mathbb{YF}$ множество всех путей из $v$ в $w$ обозначим через $T(v,w)$, а их количество через $d(v,w):=|T(v,w)|.$ В частности, $d(v,w)=0$ при $|w|<|v|.$ Также для $t\in T(v,w)$ и $n\in\overline{|v|,|w|}$ мы определяем $t(n)$ как вершину ранга $n,$ через которую проходит путь $t.$  

Кроме того, при $v,w\in{}^r\mathbb{YF},$ $t\in T(v,w)$ пусть $\overline{t}\in T(\overline{v},\overline{w})$ --- это такой путь в $\mathbb{YF},$ что для любого $m\in\overline{|v|,|w|}$ выполняется $\overline{t(m)}=\overline{t}(m).$ Из определений графов ${}^r\mathbb{YF}$ следует, что $\overline{t}$ --- это действительно путь из $\overline{v}$ в $\overline{w}$ в графе Юнга --- Фибоначчи.

Количество цифр в $w \in {}^r\mathbb{YF}\cup {}^r\mathbb{YF}_\infty$ обозначим через $\#w,$ количество единиц в $w$ --- через $e(w),$ количество двоек в $w$ --- через $d(w)$ (для бесконечных слов $w\in {}^r\mathbb{YF}_\infty$ имеем $\#w=\infty,$ но величины $d(w)$ и $e(w)$ могут быть как конечны, так и бесконечны). Очевидны соотношения
\begin{gather*}
	e(w)+d(w)=\#w,\qquad e(w)+2d(w)=|w|,\qquad \#w+d(w)=|w|. 
\end{gather*}

Кроме того, пусть вершины $v\in{}^r\mathbb{YF},$ $w,u \in {}^r\mathbb{YF}\cup{}^r\mathbb{YF}_\infty$ таковы, что $w=uv.$ Тогда, если $\#v=m,$ то положим $w[m]:=u,$ а если $v$ начинается на единицу и $e(v)=\ell,$ то положим $w\{\ell\}:=u,$ $w\langle\ell\rangle:=v.$ Также положим $w\{0\}:=w,$ $w\langle0\rangle:=\varepsilon.$ Ясно, что для данного $w$ вершины $w[m]$ определены при $m\in\overline{\#w},$ а вершины $w\{\ell\},$ $w\langle\ell\rangle$ --- при $\ell\in\overline{e(w)}.$

Также пусть $v\in{}^r\mathbb{YF},$ $w \in {}^r\mathbb{YF}\cup{}^r\mathbb{YF}_\infty.$ Определим $u\in {}^r\mathbb{YF}\cup{}^r\mathbb{YF}_\infty$ как самый длинный общий суффикс слов $v,w$, то есть самое длинное слово $u$, для которого $v=v'u$, $w=w'u$. Положим $h(v,w):=\#u$, $h'(v,w):=|u|,$ $e(v,w):= e(u),$ а также
\begin{equation*}
\mathbbm{1}_{v,w} = \begin{cases}
   1, &\text {если $v'$ и $w'$ заканчиваются на единицы разных цветов,}\\
   0, &\text{иначе.}
 \end{cases}
\end{equation*}    
Ясно, что $\mathbbm{1}_{v,w}=1\Longleftrightarrow h(v,w)<h(\overline{v},\overline{w}).$

Для перечисления путей в графах ${}^r\mathbb{YF}$ необходимо ввести следующие функции $f$ и $g$:

Функция
 \begin{gather*}
 	f(v,i,k): \left\{(v,i,k)\in\mathbb{YF}\times \mathbb{N}_0\times\mathbb{N}_0:\;i\in\overline{|v|},\;k\in\overline{\#v}\right\}\to\mathbb{Q}
 \end{gather*}
	определяется следующим образом:
	
	При $k=0:$
	\begin{itemize}
		\item Если $v\in \mathbb{YF}$ имеет суффикс с суммой цифр $i$, то есть
		$v=\alpha_1\dots\alpha_m\alpha_{m+1}\dots\alpha_n,$ где $|\alpha_{m+1}\dots\alpha_n|=i,$ $\alpha_j\in\{1,2 \},$ то
		\begin{align*}
		f(v,i,0):=&\frac{1}{(-\alpha_{m+1})(-\alpha_{m+1}-\alpha_{m+2})\dots(-\alpha_{m+1}-\dots-\alpha_{n})}\\
		&\times\frac{1}{\alpha_m(\alpha_m+\alpha_{m-1})(\alpha_m+\alpha_{m-1}+\alpha_{m-2})\dots(\alpha_m+\dots+\alpha_1)};	
		\end{align*}
		\item Если $v\in \mathbb{YF}$ не имеет суффикса с суммой цифр $i$, то 
		\begin{gather*}	
		f(v,i,0):=0.
		\end{gather*}
	\end{itemize}
	
	При $k>0$ (рекурсивное определение):
	
	\begin{align*}
		f(v1,i,k)&:= 
		\begin{cases}
			f(v1,0,0), &\text {если $i=0,$}\\
			f(v1,i,0)+f(v,i-1,k-1), &\text{если $i>0$;}
		\end{cases}\\
		f(v2,i,k)&:= 
		\begin{cases}
			0, &\text{если $i=1,$}\\
			\frac{f(v11,i,k+1)}{1-i},&\text {если $i \ne 1$.}
		\end{cases}
	\end{align*}
	Рекурсия происходит в первую очередь по количеству двоек в $v$, а при данном количестве двоек в $v$ --- по величине $k$. 

Для $w\in {}^r\mathbb{YF}\cup{}^r\mathbb{YF}_\infty$ и натурального $j\in\overline{1,d(w)}$ определим $g(w,j)$ как сумму цифр слова	$\overline{w}$ правее $j$-й справа двойки, увеличенную на $1.$ Иначе говоря, если
	\begin{gather*}
	\overline{w}=\ldots 21^{\beta_j}2\ldots21^{\beta_2}21^{\beta_1},
	\end{gather*}
	то $g(w,j)=\beta_1+\ldots+\beta_{j}+2j-1.$
	
Для количества путей между двумя вершинами в графе Юнга --- Фибоначчи известна следующая формула:
\begin{theorem}[\textup{см. \cite[теорема~1]{Evtuh2}, \cite[теорема~1]{Evtuh1}}] \label{evtuh}
Пусть $v,w \in \mathbb{YF}:$ $|w|\geqslant |v|$. Тогда
\begin{gather*}		
  d(v,w)=\sum_{i=0}^{|v|}\left( {f(v,i,h(v,w))}\prod_{j=1}^{d(w)}(g(w,j)-i)\right).  
\end{gather*}
\end{theorem}
В частности, при $v=\varepsilon$ и всех $k\in \overline{\#w}$
\begin{gather}
	\label{ochev}
	d(\varepsilon,w)=\prod_{j=1}^{d(w)}g(w,j)=f(w,0,k)\cdot |w|!.
\end{gather}
Явные формулы для значения функции $f$ в других точках в дальнейшем не используются.

\section{Число путей}
 Пусть $r\in\mathbb{N}.$ Тогда для вершин $v,w\in {}^r\mathbb{YF}$ и целого неотрицательного $m\in \overline{h(v,w)}$ положим $s(v,w,m):= d(w[m])-\#(v[m]),$ $s(v,w):=s(v,w,0).$ Также множество всех таких путей $t$ из $v$ в $w,$ что самый длинный совокупный общий суффикс вершин $t$ содержит ровно $m$ символов, обозначим через $T(v,w,m),$ а их количество через $d(v,w,m):=|T(v,w,m)|.$ Сразу очевидно, что 
\begin{gather}\label{z1}
    d(v,w)=\sum_{m=0}^{h(v,w)} d(v,w,m),\qquad d(v,w,m)=d(v[m],w[m],0).
\end{gather} 
 
\renewcommand{\labelenumi}{\arabic{enumi}$^\circ$}

\begin{Lemma}\label{pred1}
Пусть $r\in\mathbb{N}.$ Тогда для любой пары вершин $v,w\in{}^r\mathbb{YF}$ выполняется
\begin{gather}\label{e1}
  d(v,w) = \sum_{m=0}^{h(v,w)} d(\overline{v},\overline{w},m)\cdot r^{s(v,w,m)},
\end{gather}
в частности,
\begin{gather}\label{p2}
  d(\varepsilon,w)= d(\overline{\varepsilon},\overline{w})\cdot r^{d(w)}=r^{d(w)}\cdot\prod_{j=1}^{d(w)}g(w,j).  
\end{gather}
\end{Lemma}
\begin{proof}
В силу \eqref{z1} нам достаточно доказать, что при всех $r\in\mathbb{N},$ $v,w\in{}^r\mathbb{YF}$ выполняется
\begin{gather}\label{p1}
    d(v,w,0) = d(\overline{v},\overline{w},0)\cdot r^{d(w)-\#v}.
\end{gather}

Докажем, что любому пути $t'\in T(\overline{v},\overline{w},0)$ соответствуют ровно $r^{d(w)-\# v}$ путей $t\in T(v,w,0),$ таких, что $\overline{t}=t'.$

Графы ${}^r\mathbb{YF}$ устроены таким образом, что при данных условиях существует такое $n\in\overline{|w|,|v|},$ что $t'(n)$ не содержит единиц. Таким образом, при прохождении пути $t'$ ни одна единица не прошла весь путь от $\overline{v}$ до $\overline{w}$ (и при прохождении пути $t$ ни одна единица не прошла весь путь от $v$ до $w$). В частности, из этого следует, что если $d(w)<\#v,$ то обе части \eqref{p1} равны нулю, и доказывать нечего. Рассматриваем случай $d(w)\ge\#v.$

Далее заметим, что если единица вершины $\overline{w}$ при прохождении пути $t'$ образуется на $i$-м шаге, то соответствующая ей цветная единица вершины $w$ при прохождении пути $t$ образуется тоже на $i$-м шаге. Также если двойка вершины $\overline{w}$ изначально была единицей, которая при прохождении пути $t'$ превращается в двойку на $i$-м шаге, то соответствующая ей двойка вершины $w$ тоже изначально была цветной единицей, которая при прохождении пути $t$ превращается в двойку на $i$-м шаге. Кроме того, если двойка вершины $\overline{w}$ изначально была двойкой, то соответствующая ей двойка вершины $w$ тоже изначально была двойкой. И наконец, если двойка вершины $\overline{w}$ при прохождении пути $t'$ образуется как единица на $i$-м шаге, а потом на $j$-м шаге превращается в двойку, то соответствующая ей двойка вершины $w$ при прохождении пути $t$ может на $i$-м шаге превратиться в единицу любого из $r$ цветов, а потом на $j$-м шаге превращается в двойку. Несложно заметить, что двоек последнего вида ровно $d(w)-\#v.$ 
\end{proof}

Заметим, что при $m\in\overline{h(v,w)-1}$ выполняется
\begin{align}\label{e2}
 d(\overline{v},\overline{w},m)=&d(\overline{v}[m],\overline{w}[m])-d(\overline{v}[m+1],\overline{w}[m+1]).
 \end{align}

Кроме того, выполняется 
 \begin{align}
 d(\overline{v},\overline{w},h(v,w))=&d(\overline{v}[h(v,w)],\overline{w}[h(v,w)])-\mathbbm{1}_{v,w}\cdot d(\overline{v}[h(v,w)+1],\overline{w}[h(v,w)+1]).\label{e3}
\end{align}
 
Таким образом, из равенств \eqref{e1},\eqref{e2},\eqref{e3} получаем первую формулу для числа путей в графах ${}^r\mathbb{YF}$ через числа путей в графе Юнга --- Фибоначчи (формула для которых уже была приведена в теореме \ref{evtuh}):
\renewcommand{\labelenumi}{\arabic{enumi}$)$}
\begin{theorem}
Пусть $r\in\mathbb{N}.$ Тогда для любой пары вершин $v,w\in{}^r\mathbb{YF}$ выполняется
\begin{align*}
    d(v,w) =&\sum_{m=0}^{h(v,w)-1} (d(\overline{v}[m],\overline{w}[m])-d(\overline{v}[m+1],\overline{w}[m+1]))\cdot r^{s(v,w,m)}\\
    &+(d(\overline{v}[h(v,w)],\overline{w}[h(v,w)])-\mathbbm{1}_{v,w}\cdot d(\overline{v}[h(v,w)+1],\overline{w}[h(v,w)+1]))\cdot r^{s(v,w,h(v,w))}\\
    =&d(\overline{v},\overline{w})\cdot r^{s(v,w)} +\sum_{m=1}^{h(v,w)} d(\overline{v}[m],\overline{w}[m])\cdot (r^{s(v,w,m)}-r^{s(v,w,m-1)})\\
    &-\mathbbm{1}_{v,w}\cdot d(\overline{v}[h(v,w)+1],\overline{w}[h(v,w)+1])\cdot r^{s(v,w,h(v,w))}.
\end{align*}
\end{theorem}

Далее заметим, что при $m\in\overline{h(v,w)}$ выражение $s(v,w,m)$ равняется $s(v,w)+e(v)-e(v[m]).$ Из этого наблюдения следует вторая формула для числа путей в графах ${}^r\mathbb{YF}$:
\begin{theorem}\label{paths}
Пусть $r\in\mathbb{N}.$ Тогда для любой пары вершин $v,w\in{}^r\mathbb{YF}$ выполняется
\begin{multline*}
    d(v,w) = r^{s(v,w)}\cdot\Biggl (d(\overline{v},\overline{w})+ \sum_{k=1}^{e(v,w)} d(\overline{v}\{k\},\overline{w}\{k\})\cdot ( r^k-r^{k-1})-\\
     - \mathbbm{1}_{v,w}\cdot d(\overline{v}\{e(v,w)+1\},\overline{w}\{e(v,w)+1\})\cdot r^{e(v,w)} \Biggr).
\end{multline*}
\end{theorem}

Также важным следствием леммы \ref{pred1} является следующая
\begin{Lemma}\label{suuum}
Пусть $r\in\mathbb{N}.$ Тогда для любых $w\in{}^r\mathbb{YF},$ $u\in\mathbb{YF}$ выполняется
\begin{gather}\label{summa1}
    \sum_{v:\;\overline{v}=u}d(v,w)=d(u,\overline{w})\cdot r^{d(w)-d(u)},
\end{gather}
а для любых $v\in{}^r\mathbb{YF},$ $u\in\mathbb{YF}$ выполняется
\begin{gather}\label{summa3}
    \sum_{w:\;\overline{w}=u}d(v,w)=d(\overline{v},u)\cdot r^{\#u-\#v}.
\end{gather}
\end{Lemma}
\begin{proof}
Сначала докажем \eqref{summa1}. Заметим, что в силу \eqref{e1}, того, что $s(u,w,m)+e(u[m])=d(w)-d(u),$ а также \eqref{z1} левая часть \eqref{summa1} равняется
\begin{align*}
    \sum_{v:\;\overline{v}=u} \sum_{m=0}^{h(v,w)} d(u,\overline{w},m)\cdot r^{s(u,w,m)}=&\sum_{m=0}^{h(u,\overline{w})}|\{v:\;\overline{v}=u,\;h(v,w)\ge m\}|\cdot d(u,\overline{w},m)\cdot r^{s(u,w,m)}\\
    =&\sum_{m=0}^{h(u,\overline{w})} d(u,\overline{w},m)\cdot r^{d(w)-d(u)}= d(u,\overline{w})\cdot r^{d(w)-d(u)}.
\end{align*}
\eqref{summa3} получается аналогично с использованием того, что $s(v,u,m)+e(u[m])=\#u-\#v.$
\end{proof}

\section{Описание границы Мартина}

Далее нам понадобятся следующие обозначения.

При данном натуральном $r$ будем говорить, что бесконечная последовательность $\tau=\{w_m\}_{m=1}^\infty\in({}^r\mathbb{YF})^\infty$ сходится к <<бесконечно удалённой вершине>> $w\in{}^r\mathbb{YF}_\infty$ и писать  $\tau\to w,$ если ${h(w_m,w) \xrightarrow[m\to\infty]{} \infty}.$ 

Также пусть $w\in {}^r\mathbb{YF}\cup{}^r\mathbb{YF}_\infty,$ $k\in\mathbb{N}.$ Тогда
\begin{gather*}
  \pi_k(w):=\prod_{i:g(w,i)>k} \frac{g(w,i)-k}{g(w,i)}.  
\end{gather*}
Кроме того, при $k=1$ пишем $\pi(w):=\pi_1(w).$ Сразу заметим, что $\pi_k(w)=\pi_k(\overline{w}).$ Также положим
\begin{gather*}
  {}^r\mathbb{YF}_\infty^+:=\{w\in{}^r\mathbb{YF}_\infty:\;\pi(w)>0\}. 
\end{gather*}

Помимо этого, нам понадобится следующее замечание, непосредственно вытекающее из \cite[утверждение~6]{BE}:
\begin{Zam}\label{KerovGoodman}
При данном натуральном $r$ пусть бесконечная последовательность $\tau=\{w_m\}_{m=1}^\infty\in({}^r\mathbb{YF})^\infty$ сходится к <<бесконечно удалённой вершине>> $w\in{}^r\mathbb{YF}_\infty;$ кроме того, пусть $\beta\in(0,1],$ $k\in\mathbb{N}.$ Тогда
\begin{gather*}
  \lim_{m\to\infty}\frac{\pi(w_m)}{\pi(w)}=\beta  
\end{gather*}
в том и только в том случае, когда
\begin{gather*}
  \lim_{m\to\infty}\frac{\pi_k(w_m)}{\pi_k(w)}=\beta^k.
\end{gather*}    
\end{Zam}

Итак, при данном натуральном $r$ рассмотрим последовательность вершин $\tau'=\{w'_m\}_{m=1}^\infty\in({}^r\mathbb{YF})^\infty,$ такую, что $|w'_m|\to\infty.$ Если для любой вершины $v\in{}^{r}\mathbb{YF}$ существует предел
\begin{gather}\label{lim10}
  \mu_{\tau'}(v):=\lim_{m\to \infty} d(\varepsilon,v)\frac{d(v,w'_m)}{d(\varepsilon,w'_m)},  
\end{gather}
то последовательность $\tau'$ порождает такую меру $\mu_{\tau'}$ на пространстве бесконечных путей в графе ${}^{r}\mathbb{YF},$ что мера множества всех путей, проходящих через вершину $v,$ равна $\mu_{\tau'}(v),$ при этом все $d(\varepsilon,v)$ начальных отрезков пути из корня в $v$ имеют равную меру $\mu_{\tau'}(v)/d(\varepsilon,v).$ Эти меры образуют так называемую \emph{границу Мартина} графа ${}^r\mathbb{YF}.$ Также из того, что в правой части \eqref{lim10} под пределом написана доля количества путей из $T(\varepsilon,w_m'),$ проходящих через вершину $v,$ от количества всех путей из $T(\varepsilon,w'_m),$ очевидно следует, что все эти меры являются вероятностными в том смысле, что для любого $n\in\mathbb{N}_0$ выполняется соотношение
\begin{gather*}
    \sum_{v\in{}^r\mathbb{YF}_n} \mu_{\tau'}(v)=1.
\end{gather*} Для $r=1$ такие меры описаны в \cite{GK}. Наша задача --- описать все такие меры для $r\ge 2.$

Ясно, что у последовательности $\tau'$ есть бесконечная подпоследовательность $\tau=\{w_m\}_{m=1}^\infty\in({}^r\mathbb{YF})^\infty,$ такая, что $\tau\to w\in{}^r\mathbb{YF}_\infty,$ а также выполнено ровно одно из двух условий:
\begin{enumerate}
    \item \label{u1} $w\in{}^r\mathbb{YF}_\infty^+,$ и существует такое $\beta\in(0,1],$ что $\pi(w_m)/\pi(w)\to\beta;$
    \item \label{u2} $\pi(w_m)\to 0.$
\end{enumerate}
Понятно, что для описания границы Мартина достаточно рассматривать только такие последовательности $\tau$.

\renewcommand{\labelenumi}{\arabic{enumi}$)$}

Следующие лемма \ref{tgk} и теорема \ref{top} являются обобщением результата из \cite{GK}:
\begin{Lemma}\label{tgk} 
Пусть $r\in\mathbb{N}.$ Тогда для любой бесконечной последовательности вершин $\tau=\{w_m\}_{m=1}^\infty\in({}^r\mathbb{YF})^\infty,$ такой, что $\tau\to w\in{}^r\mathbb{YF}_\infty$ и выполнено одно из условий \textup{\ref{u1}),\;\ref{u2})}, предел $\mu_\tau(v)$ существует для каждой вершины $v\in{}^r\mathbb{YF},$ при этом
\begin{enumerate}
    \item Если последовательность $\tau$ удовлетворяет условию \textup{\ref{u1})}, то значение $\mu_\tau(v)$ зависит только от $v,$ $w,$ $\beta.$ При этом для любых $w\in {}^r\mathbb{YF}_\infty^+,$ $\beta\in(0,1]$ существует последовательность $\tau,$ которой соответствуют эти $w,\beta.$
    \item Если последовательность $\tau$ удовлетворяет условию \textup{\ref{u2})}, то 
    \begin{gather*}
        \mu_\tau(v)=\frac{d(\overline{\varepsilon},\overline{v})^2}{|v|!\cdot r^{e(v)}},
    \end{gather*}
то есть значение $\mu_\tau(v)$ зависит только от $v.$
\end{enumerate}
\end{Lemma}
\begin{proof}
Лемма в случае $r=1$ была доказана в \cite[Theorem~8.7]{GK}. Однако наше доказательство на это не опирается и применимо также в случае $r=1.$

Зафиксируем вершину $v\in{}^r\mathbb{YF}.$ Очевидно, что мы можем считать, что для всех $m$ выполняется $h'(w_m,w)\ge |v|.$ Итак, согласно \eqref{p2} и теореме \ref{paths} и тому, что $e(v,w_m)=e(v,w),$ $\mathbbm{1}_{v,w_m}=\mathbbm{1}_{v,w},$ выражение 
\begin{gather}\label{limit}
    \lim_{m\to \infty} d(\varepsilon,v)\cdot\frac{d(v,w_m)}{d(\varepsilon,w_m)}
\end{gather}
равняется пределу по $m$ произведения выражения 
\begin{gather}\label{mnozh}
d(\overline{\varepsilon},\overline{v})\cdot r^{d(v)}\cdot\frac{ r^{s(v,w_m)}}{ r^{d(w_m)}}=\frac{d(\overline{\varepsilon},\overline{v})}{r^{e(v)}}.
\end{gather}
и суммы с коэффициентами по $k\in\overline{e(v,w)}$ или по $k\in\overline{e(v,w)}+1$ выражений
\begin{gather}\label{slag}
    \frac{d(\overline{v}\{k\},\overline{w_m}\{k\})}{d(\overline{\varepsilon},\overline{w_m})},
\end{gather}
причём для $k=0$ коэффициент при соответствующем слагаемом равен единице.

Выражение \eqref{mnozh} зависит только от $v,$ поэтому будем рассматривать предел выражения \eqref{slag} при фиксированном $k\in\overline{e(v,w)}$ или $k\in\overline{e(v,w)}+1.$ По теореме \ref{evtuh} он равен
\begin{gather*}
\lim_{m\to\infty}\frac{\sum_{i=0}^{|{v\{k\}|}}\left({f(\overline{v}\{k\},i,h(\overline{v}\{k\},\overline{w_m}\{k\}))}\prod_{j=1}^{d(w_m\{k\})}(g(w_m\{k\},j)-i)\right)}{\prod_{j=1}^{d(w_m)}g(w_m,j)}.
\end{gather*}

Опять же, в данном выражении посмотрим на каждое слагаемое. Заметим, что $h(\overline{v}\{k\},\overline{w_m}\{k\})=h(\overline{v}\{k\},\overline{w}\{k\}),$ а также то, что при фиксированном $i\in\overline{|{v}\{k\}|}$ выполняется
\begin{align*}
\lim_{m\to\infty}\frac{\prod_{j=1}^{d(w_m\{k\})}(g(w_m\{k\},j)-i)}{ \prod_{j=1}^{d(w_m)}g(w_m,j)}=&\lim_{m\to\infty}\frac{\prod_{j=d(w_m\langle k\rangle)+1}^{d(w_m)}(g(w_m,j)-|w_m\langle k\rangle|-i)}{ \prod_{j=1}^{d(w_m)}g(w_m,j)}\\    
=&\frac{1}{\prod_{j=1}^{d(w\langle k\rangle)}g(w,j)}\cdot\lim_{m\to\infty}\prod_{j=d(w\langle k\rangle)+1}^{d(w_m)}\frac{g(w_m,j)-|w\langle k\rangle|-i}{g(w_m,j)}.
\end{align*}

Далее заметим, что выполняется предельное соотношение
\begin{gather*}
\lim_{m\to\infty}\prod_{j=d(w\langle k\rangle)+1}^{d(w_m)}\frac{g(w_m,j)-|w\langle k\rangle|-i}{g(w_m,j)}=\beta^{|w\langle k\rangle|+i}\cdot\prod_{j=d(w\langle k\rangle)+1}^{d(w)}\frac{g(w,j)-|w\langle k\rangle|-i}{g(w,j)}.
\end{gather*}
Действительно, оно следует из замечания \ref{KerovGoodman}: сомножители, для которых $g(w_m,j)\le |w\langle k\rangle|+i$, стабилизируются и совпадают с сомножителями справа, для которых $g(w,j)\le |w\langle k\rangle|+i$, а остальные сомножители составляют произведения $\pi_{|w\langle k\rangle|+i}(w_m)$ слева и $\beta^{|w\langle k\rangle|+i} \pi_{|w\langle k\rangle|+i}(w)$ справа.

Таким образом, мы доказали, что для каждого такого слагаемого существует предел по $m$ и его значение зависит только от $w,$ $\beta,$ $v.$

Для обоснования того, что при любых $w\in {}^r\mathbb{YF}_\infty^+,$ $\beta\in(0,1]$ существует последовательность $\tau,$ которой соответствуют эти $w,\beta,$ можно сослаться на то, что при произвольном $r$ это очевидно следует из данного факта в случае $r=1,$ доказанного в \cite{GK}, а можно доказать и напрямую с помощью замечания \ref{KerovGoodman} совершенно аналогично доказательству следствия 3 из \cite{BE}.

Первый пункт доказан.

Для доказательства второго пункта несложно заметить, что в случае $\pi(w_m)\to 0$ обнуляются все слагаемые, кроме слагаемого, соответствующего случаю $k=i=0,$ из чего, в силу \eqref{ochev}, следует, что выражение \eqref{limit} в данном случае равняется 
\begin{gather*}
    \frac{d(\overline{\varepsilon},\overline{v})}{r^{e(v)}}\cdot f(\overline{v},0,h(\overline{v},\overline{w})) = \frac{d(\overline{\varepsilon},\overline{v})^2}{|{v}|!\cdot r^{e(v)}}.    
\end{gather*}
Лемма доказана.
\end{proof}

Далее меры из первого пункта леммы \ref{tgk} обозначим через $\mu_{r,w,\beta},$ а из второго --- через $\mu_{r,P}.$ Также положим $\mu_{w,\beta}:=\mu_{1,w,\beta},$ $\mu_{P}:=\mu_{1,P}.$ Последняя мера называется \emph{мерой Планшереля} графа Юнга --- Фибоначчи. Явные значения для мер $\mu_{w,\beta}(v)$ были найдены в \cite{BE}, а явные значения для мер $\mu_{r,w,\beta}(v)$ могут быть вычислены из доказательства леммы \ref{tgk}, однако не будут здесь приводиться ввиду их громоздкости. 

Из всего вышеизложенного следует
\begin{theorem}\label{top}
    Для каждого натурального $r$ граница Мартина графа ${}^r\mathbb{YF}$ состоит из меры $\mu_{r,P}$ и мер вида $\mu_{r,w,\beta}$ при всех $w\in {}^r\mathbb{YF}_\infty^+,$ $\beta\in (0,1].$
\end{theorem}

Различность таких мер при $r=1$ была доказана в \cite[Theorem~8.7]{GK}, а их эргодичность --- в \cite[теорема~4]{BE}. Различность данных мер в общем случае пока не доказана, но следует из их эргодичности, которая будет доказана в следующем разделе.

\renewcommand{\labelenumi}{\arabic{enumi}$^\circ$}

Таким образом, при данных $r\in\mathbb{N},$ $w\in {}^r\mathbb{YF}_\infty^+,$ $\beta\in (0,1]$ и при целом неотрицательном $n$ функции $\mu_{r,w,\beta}$ задают вероятностные меры на множествах ${}^r\mathbb{YF}_n.$ В этом смысле понимаются выражения типа $\mu_{r,w,\beta}(A)$ для подмножеств $A\subseteq{}^r\mathbb{YF}_n.$

Следующее предложение сыграет важную роль в доказательстве эргодичности границы Мартина графов ${}^r\mathbb{YF}.$
\begin{Prop}\label{suum}
Пусть $r\in\mathbb{N},$ $w\in{}^r\mathbb{YF}_\infty^+,$ $\beta\in(0,1].$ Тогда для любого множества $A\subseteq{}^1\mathbb{YF}_n$ при $n\in\mathbb{N}_0$ выполняется
\begin{gather*}
\sum_{v:\; \overline{v}\in A} \mu_{r,w,\beta}(v) = \mu_{\overline{w},\beta}(A).  
\end{gather*}
\end{Prop}
\begin{proof}
Очевидно, что если последовательность вершин $\{w_m\}_{m=1}^\infty\in({}^r\mathbb{YF})^\infty$ задаёт меру $\mu_{r,w,\beta},$ то последовательность $\{\overline{w_m}\}_{m=1}^\infty\in(\mathbb{YF})^\infty$ задаёт меру $\mu_{\overline{w},\beta}.$ Также ясно, что предложение достаточно доказать только для одноэлементных множеств $A.$ Для таких множеств предложение следует из \eqref{summa1} и \eqref{p2}. 
\end{proof}

\section{Эргодичность границы Мартина}
Введём дополнительные обозначения:

Пусть $r\in\mathbb{N},$ $w\in{}^r\mathbb{YF}_\infty,$ $\beta\in(0,1],$ $n,\ell\in\mathbb{N}_0,$ $\varepsilon>0,$ $k\in\mathbb{N}.$ Положим
\begin{align*}
R(w,\beta,n,\varepsilon):=&\{v\in{}^r\mathbb{YF}_n:\quad \pi(v)\in(\pi(w)(\beta-\varepsilon),\pi(w)(\beta+\varepsilon)) \};\\
    \overline{R}(w,\beta,n,\varepsilon):=&\{v\in{}^r\mathbb{YF}_n:\quad \pi(v)\notin(\pi(w)(\beta-\varepsilon),\pi(w)(\beta+\varepsilon)) \};\\
    {Q}(w,n,\ell):=&\{v\in{}^r\mathbb{YF}_n:\quad h'(v,w)\ge \ell\};\\
    \overline{Q}(w,n,\ell):=&\{v\in{}^r\mathbb{YF}_n:\quad h'(v,w)< \ell\};\\
    \widetilde{Q}(w,n,\ell):=&\{v\in{}^r\mathbb{YF}_n:\quad h'(\overline{v},\overline{w})< \ell\};\\
    \widehat{Q}(w,n,\ell,k):=&\{v\in{}^r\mathbb{YF}_n:\quad h'(\overline{v},\overline{w})\ge \ell,\;e(v,w)=k-1 \}.
\end{align*}

Также пусть $r\in\mathbb{N},$ $w\in{}^r\mathbb{YF}\cup{}^r\mathbb{YF}_\infty.$ Тогда для каждого $k\in\overline{1,e(w)}$ определим вершину $w(k)\in{}^r\mathbb{YF}\cup{}^r\mathbb{YF}_\infty,$ получающуюся из $w$ путём замены $k$-й справа единицы на двойку.

Далее обобщим результат из \cite{BE}:
\begin{theorem} \label{t3}
Пусть $r\in\mathbb{N},$ $w\in{}^r\mathbb{YF}_\infty^+,$ $\beta\in(0,1],$ зафиксируем $\ell\in \mathbb{N}_0$ и $\varepsilon>0.$ Тогда мера $\mu_{r,w,\beta}$ при больших $n$ сосредоточена на $R(w,\beta,n,\varepsilon)\cap Q(w,n,\ell):$
\begin{gather*}
  \lim_{n\to \infty} \mu_{r,w,\beta}(\overline{R}(w,\beta,n,\varepsilon)\cup \overline{Q}(w,n,\ell))=0.
\end{gather*}
\end{theorem}
\begin{proof}
В случае $r=1$ данная теорема была доказана в \cite[следствие~4 и теорема~3]{BE}. Перейдём к случаю $r\ge2.$

Пусть самый длинный суффикс вершины $w$ из тех, ранг которых не превосходит $\ell,$ содержит $e$ единиц. Заметим, что 
\begin{gather*}
    \overline{Q}(w,n,\ell)=\widetilde{Q}(w,n,\ell)\cup\bigcup_{k=1}^{e} \widehat{Q}(w,n,\ell,k).
\end{gather*}

Из предложения \ref{suum} и случая $r=1$ очевидно, что
\begin{gather}\label{lim1}
    \lim_{n\to\infty} \mu_{r,w,\beta}(\overline{R}(w,\beta,n,\varepsilon)\cup\widetilde{Q}(w,n,\ell)) =\lim_{n\to\infty}\mu_{\overline{w},\beta}(\overline{R}(\overline{w},\beta,n,\varepsilon)\cup \overline{Q}(\overline{w},n,\ell))=0.
\end{gather}

Остались множества $\widehat{Q}(w,n,\ell,k).$ Зафиксируем $k\in\overline{1,e}.$ Рассмотрим такую последовательность $\tau=\{w_m\}_{m=1}^\infty,$ порождающую меру $\mu_{r,w,\beta},$ что для всех $m$ выполняется $h'(w_m,w)\ge \ell.$ Кроме того, рассмотрим последовательность вершин $\{w_m(k)\}_{m=1}^\infty,$ сходящуюся к $w(k).$ Положим
\begin{gather*}
    \beta(k):=\lim_{m\to \infty}\frac{\pi(w_m(k))}{\pi(w(k))}.
\end{gather*}

Оценим величину $\beta(k).$ Очевидно, что для всех $m\in\mathbb{N}$
\begin{gather*}
    \frac{\pi(w_m(k))}{\pi(w_m)}\ge \frac{1}{2},
\end{gather*}
из чего следует, что $\beta(k)\in(0,1].$ Также очевидно включение $\widehat{Q}(w,n,\ell,k)\subseteq \widetilde{Q}(w(k),n,\ell).$ Поэтому в силу того, что \eqref{lim1} выполняется и для $\mu_{r,w(k),\beta(k)}(\widetilde{Q}(w(k),n,\ell)),$ для завершения доказательства теоремы достаточно доказать, что существует константа $C,$ такая, что для любых $v\in \widehat{Q}(w,n,\ell,k),$ $m\in\mathbb{N}$ выполняется
\begin{gather*}
\frac{d(v,w_m)}{d(\varepsilon,w_m)}\le C\cdot\frac{d(v,w_m(k))}{d(\varepsilon,w_m(k))}.
\end{gather*}

Для начала заметим, что, согласно \eqref{p2} и определению функции $g,$ при всех натуральных $m$ выполняется
\begin{gather*}
    \frac{d(\varepsilon,w_m(k))}{d(\varepsilon,w_m)}\le r\cdot \ell\cdot \prod_{j=1}^{d(w_m)}\frac{g(w_m,j)+1}{g(w_m,j)}\le \frac{2\cdot r\cdot \ell}{\pi(w_m)}\xrightarrow[m\to\infty]{} \frac{2\cdot r\cdot \ell}{
    \beta\cdot\pi(w)}.
\end{gather*}

Таким образом, для доказательства теоремы достаточно доказать, что при любых $v\in \widehat{Q}(w,n,\ell,k),$ $m\in\mathbb{N}$ выполняется $d(v,w_m)\le d(v,w_m(k)).$ Рассмотрим путь $t\in T(v,w_m).$ Заметим, что $k$-я справа единица вершины $w_m$ не могла изначально содержаться в $v.$ Инъективно сопоставим этому пути путь $t'\in T(v,w_m(k))$ следующим образом: после того шага пути $t,$ на котором образуется $k$-я справа единица вершины $w_m$, добавим дополнительный шаг, на котором эта единица сразу превратится в двойку, а остальные шаги оставим без изменений.

Теорема доказана.
\end{proof}

\begin{theorem}
Пусть $r\in\mathbb{N}.$ Тогда любая мера границы Мартина графа ${}^r\mathbb{YF}$ эргодична, то есть не является смесью других мер такого вида.
\end{theorem}
\begin{proof}
Эргодичность $\mu_P$ была доказана в \cite{KerGned}, и доказательство для $\mu_{r,P}$ вполне аналогично. Эргодичность мер $\mu_{w,\beta}$ была доказана в \cite[теорема~4]{BE}, и доказательство для $\mu_{r,w,\beta}$ с помощью теоремы \ref{t3} совершенно аналогично.
\end{proof}

\section{Благодарности}
Я очень признателен моему научному руководителю Фёдору Владимировичу Петрову за постановку задачи, проявленное терпение при проверке доказательства и редактировании статьи, а также за моральную поддержку на протяжении всего периода работы. 

Работа выполнена при поддержке СПбГУ, шифр проекта 159547043.

{}

\selectlanguage{english}

\title{The Martin boundary of the $r$-differential version of the Young--Fibonacci graph}

\author{V.~Yu.~Evtushevsky}

\maketitle

\setcounter{page}{1}

\etocsettagdepth{en}{all}
\etocsettagdepth{ru}{none}
\tableofcontents

\theoremstyle{style2}
\newtheorem{ENLemma}{Lemma}

\theoremstyle{style2}
\newtheorem{ENAlg}{Algorithm}

\theoremstyle{style2}
\newtheorem{ENCol}{Corollary}

\theoremstyle{style2}
\newtheorem{ENtheorem}{Theorem}

\theoremstyle{style2}
\newtheorem{ENDef}{Definition}

\theoremstyle{style2}
\newtheorem{ENProp}{Proposition}

\theoremstyle{style2}
\newtheorem{ENProblem}{Problem}

\theoremstyle{style2}
\newtheorem{ENZam}{Remark}

\theoremstyle{style2}
\newtheorem{ENOboz}{Notation}

\theoremstyle{style2}
\newtheorem{ENEx}{Example}

\theoremstyle{style2}
\newtheorem{ENNab}{Observation}

\setcounter{section}{0}
\setcounter{figure}{0} 
\setcounter{equation}{0}

\begin{abstract}
For the $r$-differential version of the Young--Fibonacci graph, formulas for the number of paths between vertices are obtained; the Martin boundary of the path space is described; and the ergodicity of the corresponding measures is proved.
\end{abstract}

\etocdepthtag.toc{en}
\section{Introduction}

For a positive integer $r$, consider words over an alphabet
of size $r+1$, consisting of a two and $r$ ones of different colors,
with a given digit sum $n.$ As is well known, their number is the $r$-Fibonacci number $F_{r,n+1}$ ($F_{r,0}=0,$ $F_{r,1}=1,$ $F_{r,k+2}=r\cdot F_{r,k+1}+F_{r,k}$), and this is the most common combinatorial interpretation of $r$-Fibonacci numbers. One can also think of tilings of a $2\times n$ strip by $1\times 2$ and $2\times 1$ dominoes,
where the vertical dominoes can be colored in $r$ colors
(twos correspond to pairs of horizontal dominoes, and ones to vertical dominoes.)

For each positive integer $r$, introduce a partial order on the corresponding set of words: we say that a word $v$ \emph{precedes} a word $w$ if, after removing the longest common suffix, the number of twos remaining in $w$ is at least the number of symbols remaining in $v$.

These are indeed partial orders; moreover, the corresponding partially ordered sets are modular lattices. For $r=1$, such a lattice is known as the Young--Fibonacci lattice, and for arbitrary $r$, as the $r$-differential version of the Young--Fibonacci lattice.

\begin{figure}[htbp]
    \centering
    \includegraphics[width=12cm, height=10cm]{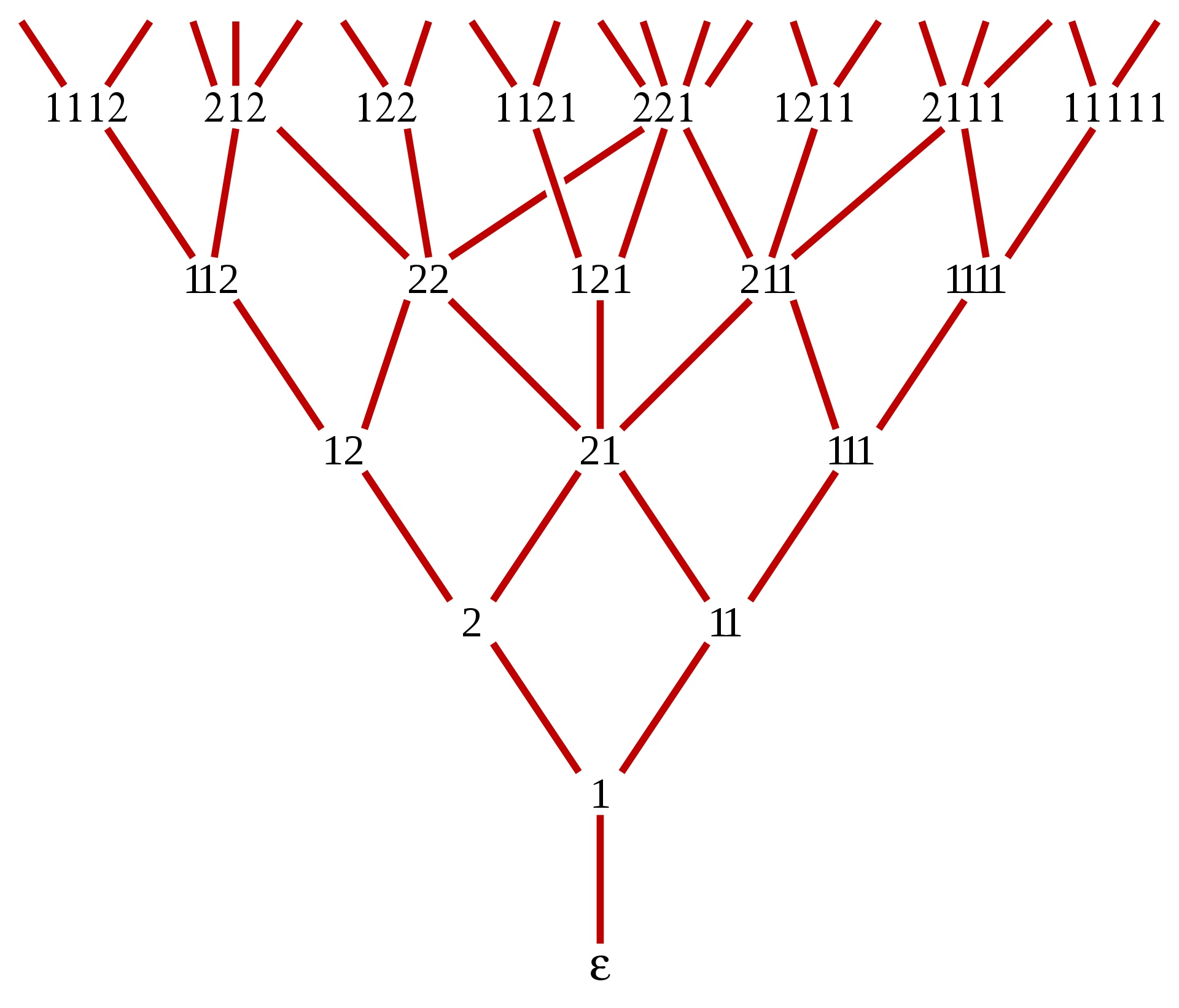}
    \caption{The Young--Fibonacci graph}
    \label{ENyf}
\end{figure}

\renewcommand{\labelenumi}{\arabic{enumi}.}

The Hasse diagrams of these lattices are called the Young--Fibonacci graph (shown in Fig. \ref{ENyf}) and the $r$-differential version of the Young--Fibonacci graph, respectively. These are graded graphs, which we represent as growing from bottom to top, starting from the empty word. The grading is given by the digit sum function (we assume that a one
of any color contributes $1$ to the digit sum). Let us explicitly describe the directed edges. The ``upward'' edges from a given word $v$ lead to words obtained from $v$ by one of two operations:
\begin{enumerate}
    \item replace the leftmost one (of any color) by a two;
    \item insert a one (of any color) at a position
    with no ones to its left.
\end{enumerate}

In this paper, we regard these graphs as directed and containing only ``upward'' edges.

For a fixed $r$, such a graph, in addition to being modular, is $r$-differential, that is, at each vertex the out-degree exceeds the in-degree by $r$.

The study of the graded Young--Fibonacci graph was initiated in 1988 simultaneously and independently by Richard Stanley \cite{ENSt} and Sergei Vladimirovich Fomin \cite{ENFo}.

One reason for the interest in this graph is that there are only two rooted locally finite $1$-differential modular lattices; the other is the lattice of Young diagrams, which plays a key role in the representation theory of the symmetric group \cite[Theorem~1.3]{ENBur}.

The key questions about graded graphs concern central measures on the space of (infinite) paths in the graph. This viewpoint was systematically developed in the works of Anatoly Moiseevich Vershik, to whose recent survey \cite{ENVer} and the literature cited therein we refer the reader.

Among central measures, one singles out those that are limits of measures induced by paths to distant vertices---the so-called \emph{Martin boundary} of the graph.

The boundary of the path space of the Young--Fibonacci graph was studied by Frederick Goodman and Sergei Vasilievich Kerov (2000) \cite{ENGK}. They used Okada's algebraic formalism \cite{ENOkada}.

As follows from the definition itself, the asymptotic question about the boundary is directly related to the enumerative question about the number of paths between two vertices of the graph. We note Fomin's important general work \cite{ENFo1} on the enumeration of paths in graded graphs, which gives a number of general identities and points out a connection, among other things, with a generalization of the Robinson--Schensted--Knuth algorithm.

In the work of Alexander Vasilievich Gnedin and S.~V.~Kerov \cite{ENKerGned}, the ergodicity of the so-called Plancherel measure of the Young--Fibonacci graph was proved. In the author's work \cite{ENEvtuh1} (a short version of which was published as \cite{ENEvtuh2}), an explicit formula was proposed for the number of paths between two vertices in the Young--Fibonacci graph, with
a number of summands at most quadratic in the length of the shorter word. In the work of Ivan Alekseevich Bochkov and the author \cite{ENBE}, for all measures of the Martin boundary of the Young--Fibonacci graph, an explicit formula was found for the measure of the set of all infinite paths from the root passing through a given vertex, after which the ergodicity of these measures was proved.

The aim of the present paper is to extend the above results of \cite{ENGK, ENKerGned, ENEvtuh1, ENEvtuh2, ENBE} to the case of arbitrary $r.$

In Section 3, we propose explicit formulas for the number of paths between two vertices in the $r$-differential version of the Young--Fibonacci graph, with
a number of summands at most cubic in the length of the shorter word. In Section 5, we completely describe the Martin boundary of this graph. For each measure of this boundary, an explicit formula for the measure of the set of all infinite paths from the root passing through a given vertex follows from the proof, but is omitted because of its cumbersome form. In Section 6, the ergodicity of all these measures is proved. As usual, Vershik's \emph{ergodic method} is used to prove ergodicity.

\section{Basic concepts and notation}

\renewcommand{\labelitemii}{$\bullet$} 

The sets of nonnegative integers and positive integers (natural numbers) are denoted by $\mathbb{N}_0$ and $\mathbb{N}$, respectively.
For $m,n\in \mathbb{N}_0:$ $m\leqslant n$, put
\begin{gather*}
	\overline{m,n}:=\{m,m+1,\ldots,n\},\quad \overline{n}:=\overline{0,n}=\{0,1,\ldots,n\},\quad \overline{m,\infty}:=\{m,m+1,\ldots\}.
\end{gather*}

Consider the alphabet $\mathcal{A}_r$ consisting of a two and $r$ ones of different colors.
For a fixed $r\in\mathbb{N}$, let ${}^r\mathbb{YF}$ be the $r$-differential version of the Young--Fibonacci graph; its vertices are finite words over the alphabet $\mathcal{A}_r$. In what follows, ${}^r\mathbb{YF}$, when regarded as a set, is precisely the set of vertices of this graph. In addition, we also denote the Young--Fibonacci graph ${}^1\mathbb{YF}$ by $\mathbb{YF}$ and regard it as a graph on the set of finite words in $1$ and $2$.

If $r\in\mathbb{N},$ and $\{\alpha_i\}_{i=1}^\infty\in\mathcal{A}_r^\infty$ is an \emph{infinite} sequence of twos and colored ones, associate with it a ``vertex at infinity'' $w=\ldots\alpha_2\alpha_1$ of the graph ${}^r\mathbb{YF}.$ Denote the sets of ``vertices at infinity'' of the graphs ${}^r\mathbb{YF}$ and $\mathbb{YF}$ by ${}^r\mathbb{YF}_\infty$ and $\mathbb{YF}_\infty$, respectively. For $w\in{}^r\mathbb{YF}\cup{}^r\mathbb{YF}_\infty,$ $v \in {}^r\mathbb{YF}$, the word $wv$ is defined as the concatenation of $w$ and $v$. Also, for $v\in{}^r\mathbb{YF}$ and $n\in\mathbb{N}_0$, the word $v^n$ is defined in the natural way.

For $v \in {}^r\mathbb{YF}$, denote the digit sum of the word $v$ by $|v|$. The number $|v|$ is the \emph{rank} of the vertex $v$ in the graph ${}^r\mathbb{YF}$ (we assume that the empty word has rank $0$ and that moving upward along an edge increases the rank by $1$). For $w \in {}^r\mathbb{YF}_\infty$, put $|w|=\infty.$

For a nonnegative integer $n$, denote the set of vertices of the graph ${}^r\mathbb{YF}$ of rank $n$ by
\begin{gather*}
	{}^r\mathbb{YF}_n:=\{v\in{}^r\mathbb{YF}:\;|v|=n\}.
\end{gather*}	
In this paper, we regard the graphs ${}^r\mathbb{YF}$ as directed, with edges from vertices of lower rank to vertices of higher rank.

In addition, associate with each vertex $w\in{}^r\mathbb{YF}\cup{}^r\mathbb{YF}_\infty$ the vertex $\overline{w}\in\mathbb{YF}\cup\mathbb{YF}_\infty,$ obtained by forgetting the colors of the ones.

For $v,w \in {}^r\mathbb{YF}$, denote the set of all paths from $v$ to $w$ by $T(v,w)$, and their number by $d(v,w):=|T(v,w)|.$ In particular, $d(v,w)=0$ when $|w|<|v|.$ Also, for $t\in T(v,w)$ and $n\in\overline{|v|,|w|}$, we define $t(n)$ as the vertex of rank $n,$ visited by the path $t.$

In addition, for $v,w\in{}^r\mathbb{YF},$ $t\in T(v,w)$, let $\overline{t}\in T(\overline{v},\overline{w})$ be the path in $\mathbb{YF},$ such that, for every $m\in\overline{|v|,|w|}$, we have $\overline{t(m)}=\overline{t}(m).$ It follows from the definitions of the graphs ${}^r\mathbb{YF}$ that $\overline{t}$ is indeed a path from $\overline{v}$ to $\overline{w}$ in the Young--Fibonacci graph.

Denote the number of digits in $w \in {}^r\mathbb{YF}\cup {}^r\mathbb{YF}_\infty$ by $\#w,$ the number of ones in $w$ by $e(w),$ and the number of twos in $w$ by $d(w)$ (for infinite words $w\in {}^r\mathbb{YF}_\infty$, we have $\#w=\infty,$ but the quantities $d(w)$ and $e(w)$ may be either finite or infinite). The following relations are obvious:
\begin{gather*}
	e(w)+d(w)=\#w,\qquad e(w)+2d(w)=|w|,\qquad \#w+d(w)=|w|. 
\end{gather*}

In addition, let the vertices $v\in{}^r\mathbb{YF},$ $w,u \in {}^r\mathbb{YF}\cup{}^r\mathbb{YF}_\infty$ satisfy $w=uv.$ Then, if $\#v=m,$ put $w[m]:=u,$ and if $v$ begins with a one and $e(v)=\ell,$ put $w\{\ell\}:=u,$ $w\langle\ell\rangle:=v.$ Also put $w\{0\}:=w,$ $w\langle0\rangle:=\varepsilon.$ Clearly, for a given $w$, the vertices $w[m]$ are defined for $m\in\overline{\#w},$ and the vertices $w\{\ell\},$ $w\langle\ell\rangle$ for $\ell\in\overline{e(w)}.$

Also, let $v\in{}^r\mathbb{YF},$ $w \in {}^r\mathbb{YF}\cup{}^r\mathbb{YF}_\infty.$ Define $u\in {}^r\mathbb{YF}\cup{}^r\mathbb{YF}_\infty$ as the longest common suffix of the words $v,w$, that is, the longest word $u$ for which $v=v'u$, $w=w'u$. Put $h(v,w):=\#u$, $h'(v,w):=|u|,$ $e(v,w):= e(u),$ and also
\begin{equation*}
\mathbbm{1}_{v,w} = \begin{cases}
   1, &\text {if $v'$ and $w'$ end in ones of different colors,}\\
   0, &\text{otherwise.}
 \end{cases}
\end{equation*}    
Clearly, $\mathbbm{1}_{v,w}=1\Longleftrightarrow h(v,w)<h(\overline{v},\overline{w}).$

To enumerate paths in the graphs ${}^r\mathbb{YF}$, we need to introduce the following functions $f$ and $g$:

The function
 \begin{gather*}
 	f(v,i,k): \left\{(v,i,k)\in\mathbb{YF}\times \mathbb{N}_0\times\mathbb{N}_0:\;i\in\overline{|v|},\;k\in\overline{\#v}\right\}\to\mathbb{Q}
 \end{gather*}
    is defined as follows:
	
    For $k=0:$
	\begin{itemize}
        \item If $v\in \mathbb{YF}$ has a suffix with digit sum $i$, that is,
        $v=\alpha_1\dots\alpha_m\alpha_{m+1}\dots\alpha_n,$ where $|\alpha_{m+1}\dots\alpha_n|=i,$ $\alpha_j\in\{1,2 \},$ then
		\begin{align*}
		f(v,i,0):=&\frac{1}{(-\alpha_{m+1})(-\alpha_{m+1}-\alpha_{m+2})\dots(-\alpha_{m+1}-\dots-\alpha_{n})}\\
		&\times\frac{1}{\alpha_m(\alpha_m+\alpha_{m-1})(\alpha_m+\alpha_{m-1}+\alpha_{m-2})\dots(\alpha_m+\dots+\alpha_1)};	
		\end{align*}
        \item If $v\in \mathbb{YF}$ has no suffix with digit sum $i$, then
		\begin{gather*}	
		f(v,i,0):=0.
		\end{gather*}
	\end{itemize}
	
    For $k>0$ (recursive definition):
	
	\begin{align*}
		f(v1,i,k)&:= 
		\begin{cases}
            f(v1,0,0), &\text {if $i=0,$}\\
            f(v1,i,0)+f(v,i-1,k-1), &\text{if $i>0$;}
		\end{cases}\\
		f(v2,i,k)&:= 
		\begin{cases}
            0, &\text{if $i=1,$}\\
            \frac{f(v11,i,k+1)}{1-i},&\text {if $i \ne 1$.}
		\end{cases}
	\end{align*}
    The recursion is primarily on the number of twos in $v$, and, for a given number of twos in $v$, on the value of $k$.

For $w\in {}^r\mathbb{YF}\cup{}^r\mathbb{YF}_\infty$ and a natural number $j\in\overline{1,d(w)}$, define $g(w,j)$ as the digit sum of the word $\overline{w}$ to the right of the $j$th two from the right, increased by $1.$ In other words, if
	\begin{gather*}
	\overline{w}=\ldots 21^{\beta_j}2\ldots21^{\beta_2}21^{\beta_1},
	\end{gather*}
    then $g(w,j)=\beta_1+\ldots+\beta_{j}+2j-1.$
	
The following formula is known for the number of paths between two vertices in the Young--Fibonacci graph:
\begin{ENtheorem}[\textup{see \cite[Theorem~1]{ENEvtuh2}, \cite[Theorem~1]{ENEvtuh1}}] \label{ENevtuh}
Let $v,w \in \mathbb{YF}:$ $|w|\geqslant |v|$. Then
\begin{gather*}		
  d(v,w)=\sum_{i=0}^{|v|}\left( {f(v,i,h(v,w))}\prod_{j=1}^{d(w)}(g(w,j)-i)\right).  
\end{gather*}
\end{ENtheorem}
In particular, for $v=\varepsilon$ and all $k\in \overline{\#w}$,
\begin{gather}
	\label{ENochev}
	d(\varepsilon,w)=\prod_{j=1}^{d(w)}g(w,j)=f(w,0,k)\cdot |w|!.
\end{gather}
Explicit formulas for the values of the function $f$ at other points are not used in what follows.

\section{The number of paths}
 Let $r\in\mathbb{N}.$ Then, for vertices $v,w\in {}^r\mathbb{YF}$ and a nonnegative integer $m\in \overline{h(v,w)}$, put $s(v,w,m):= d(w[m])-\#(v[m]),$ $s(v,w):=s(v,w,0).$ Also, denote the set of all paths $t$ from $v$ to $w,$ such that the longest suffix common to all vertices of $t$ contains exactly $m$ symbols, by $T(v,w,m),$ and their number by $d(v,w,m):=|T(v,w,m)|.$ It is immediately clear that
\begin{gather}\label{ENz1}
    d(v,w)=\sum_{m=0}^{h(v,w)} d(v,w,m),\qquad d(v,w,m)=d(v[m],w[m],0).
\end{gather} 
 
\renewcommand{\labelenumi}{\arabic{enumi}$^\circ$}

\begin{ENLemma}\label{ENpred1}
Let $r\in\mathbb{N}.$ Then, for any pair of vertices $v,w\in{}^r\mathbb{YF}$, we have
\begin{gather}\label{ENe1}
  d(v,w) = \sum_{m=0}^{h(v,w)} d(\overline{v},\overline{w},m)\cdot r^{s(v,w,m)},
\end{gather}
in particular,
\begin{gather}\label{ENp2}
  d(\varepsilon,w)= d(\overline{\varepsilon},\overline{w})\cdot r^{d(w)}=r^{d(w)}\cdot\prod_{j=1}^{d(w)}g(w,j).  
\end{gather}
\end{ENLemma}
\begin{proof}
By \eqref{ENz1}, it suffices to prove that, for all $r\in\mathbb{N},$ $v,w\in{}^r\mathbb{YF}$,
\begin{gather}\label{ENp1}
    d(v,w,0) = d(\overline{v},\overline{w},0)\cdot r^{d(w)-\#v}.
\end{gather}

We prove that every path $t'\in T(\overline{v},\overline{w},0)$ corresponds to exactly $r^{d(w)-\# v}$ paths $t\in T(v,w,0),$ such that $\overline{t}=t'.$

The graphs ${}^r\mathbb{YF}$ are constructed so that, under these conditions, there exists $n\in\overline{|w|,|v|},$ such that $t'(n)$ contains no ones. Thus, when traversing the path $t'$, none of the ones survives the entire path from $\overline{v}$ to $\overline{w}$ (and when traversing the path $t$, none of the ones survives the entire path from $v$ to $w$). In particular, this implies that, if $d(w)<\#v,$ both sides of \eqref{ENp1} are zero, and there is nothing to prove. We consider the case $d(w)\ge\#v.$

Next, note that, if a one of the vertex $\overline{w}$, when traversing the path $t'$, is created at the $i$th step, the corresponding colored one of the vertex $w$, when traversing the path $t$, is also created at the $i$th step. Also, if a two of the vertex $\overline{w}$ was initially a one that, when traversing the path $t'$, turns into a two at the $i$th step, the corresponding two of the vertex $w$ was also initially a colored one that, when traversing the path $t$, turns into a two at the $i$th step. In addition, if a two of the vertex $\overline{w}$ was initially a two, the corresponding two of the vertex $w$ was also initially a two. Finally, if a two of the vertex $\overline{w}$, when traversing the path $t'$, is created as a one at the $i$th step and then turns into a two at the $j$th step, the corresponding two of the vertex $w$, when traversing the path $t$, can turn into a one at the $i$th step in any of the $r$ colors and then turns into a two at the $j$th step. It is easy to see that the number of twos of the last type is exactly $d(w)-\#v.$
\end{proof}

Note that, for $m\in\overline{h(v,w)-1}$,
\begin{align}\label{ENe2}
 d(\overline{v},\overline{w},m)=&d(\overline{v}[m],\overline{w}[m])-d(\overline{v}[m+1],\overline{w}[m+1]).
 \end{align}

In addition, we have
 \begin{align}
 d(\overline{v},\overline{w},h(v,w))=&d(\overline{v}[h(v,w)],\overline{w}[h(v,w)])-\mathbbm{1}_{v,w}\cdot d(\overline{v}[h(v,w)+1],\overline{w}[h(v,w)+1]).\label{ENe3}
\end{align}
 
Thus, from the equalities \eqref{ENe1},\eqref{ENe2},\eqref{ENe3}, we obtain the first formula for the number of paths in the graphs ${}^r\mathbb{YF}$ in terms of the numbers of paths in the Young--Fibonacci graph (a formula for which has already been given in Theorem \ref{ENevtuh}):
\renewcommand{\labelenumi}{\arabic{enumi}$)$}
\begin{ENtheorem}
Let $r\in\mathbb{N}.$ Then, for any pair of vertices $v,w\in{}^r\mathbb{YF}$, we have
\begin{align*}
    d(v,w) =&\sum_{m=0}^{h(v,w)-1} (d(\overline{v}[m],\overline{w}[m])-d(\overline{v}[m+1],\overline{w}[m+1]))\cdot r^{s(v,w,m)}\\
    &+(d(\overline{v}[h(v,w)],\overline{w}[h(v,w)])-\mathbbm{1}_{v,w}\cdot d(\overline{v}[h(v,w)+1],\overline{w}[h(v,w)+1]))\cdot r^{s(v,w,h(v,w))}\\
    =&d(\overline{v},\overline{w})\cdot r^{s(v,w)} +\sum_{m=1}^{h(v,w)} d(\overline{v}[m],\overline{w}[m])\cdot (r^{s(v,w,m)}-r^{s(v,w,m-1)})\\
    &-\mathbbm{1}_{v,w}\cdot d(\overline{v}[h(v,w)+1],\overline{w}[h(v,w)+1])\cdot r^{s(v,w,h(v,w))}.
\end{align*}
\end{ENtheorem}

Next, note that, for $m\in\overline{h(v,w)}$, the expression $s(v,w,m)$ equals $s(v,w)+e(v)-e(v[m]).$ This observation yields the second formula for the number of paths in the graphs ${}^r\mathbb{YF}$:
\begin{ENtheorem}\label{ENpaths}
Let $r\in\mathbb{N}.$ Then, for any pair of vertices $v,w\in{}^r\mathbb{YF}$, we have
\begin{multline*}
    d(v,w) = r^{s(v,w)}\cdot\Biggl (d(\overline{v},\overline{w})+ \sum_{k=1}^{e(v,w)} d(\overline{v}\{k\},\overline{w}\{k\})\cdot ( r^k-r^{k-1})-\\
     - \mathbbm{1}_{v,w}\cdot d(\overline{v}\{e(v,w)+1\},\overline{w}\{e(v,w)+1\})\cdot r^{e(v,w)} \Biggr).
\end{multline*}
\end{ENtheorem}

The following lemma is also an important consequence of Lemma \ref{ENpred1}.
\begin{ENLemma}\label{ENsuuum}
Let $r\in\mathbb{N}.$ Then, for any $w\in{}^r\mathbb{YF},$ $u\in\mathbb{YF}$, we have
\begin{gather}\label{ENsumma1}
    \sum_{v:\;\overline{v}=u}d(v,w)=d(u,\overline{w})\cdot r^{d(w)-d(u)},
\end{gather}
and, for any $v\in{}^r\mathbb{YF},$ $u\in\mathbb{YF}$, we have
\begin{gather}\label{ENsumma3}
    \sum_{w:\;\overline{w}=u}d(v,w)=d(\overline{v},u)\cdot r^{\#u-\#v}.
\end{gather}
\end{ENLemma}
\begin{proof}
First, we prove \eqref{ENsumma1}. Note that, by \eqref{ENe1}, the fact that $s(u,w,m)+e(u[m])=d(w)-d(u),$ and \eqref{ENz1}, the left-hand side of \eqref{ENsumma1} equals
\begin{align*}
    \sum_{v:\;\overline{v}=u} \sum_{m=0}^{h(v,w)} d(u,\overline{w},m)\cdot r^{s(u,w,m)}=&\sum_{m=0}^{h(u,\overline{w})}|\{v:\;\overline{v}=u,\;h(v,w)\ge m\}|\cdot d(u,\overline{w},m)\cdot r^{s(u,w,m)}\\
    =&\sum_{m=0}^{h(u,\overline{w})} d(u,\overline{w},m)\cdot r^{d(w)-d(u)}= d(u,\overline{w})\cdot r^{d(w)-d(u)}.
\end{align*}
\eqref{ENsumma3} is obtained similarly, using the fact that $s(v,u,m)+e(u[m])=\#u-\#v.$
\end{proof}

\section{Description of the Martin boundary}

We will need the following notation.

For a given natural number $r$, we say that an infinite sequence $\tau=\{w_m\}_{m=1}^\infty\in({}^r\mathbb{YF})^\infty$ converges to a ``vertex at infinity'' $w\in{}^r\mathbb{YF}_\infty$ and write $\tau\to w,$ if ${h(w_m,w) \xrightarrow[m\to\infty]{} \infty}.$

Also, let $w\in {}^r\mathbb{YF}\cup{}^r\mathbb{YF}_\infty,$ $k\in\mathbb{N}.$ Then
\begin{gather*}
  \pi_k(w):=\prod_{i:g(w,i)>k} \frac{g(w,i)-k}{g(w,i)}.  
\end{gather*}
In addition, for $k=1$, we write $\pi(w):=\pi_1(w).$ Note immediately that $\pi_k(w)=\pi_k(\overline{w}).$ Also put
\begin{gather*}
  {}^r\mathbb{YF}_\infty^+:=\{w\in{}^r\mathbb{YF}_\infty:\;\pi(w)>0\}. 
\end{gather*}

We will also need the following remark, which follows directly from \cite[Proposition~6]{ENBE}:
\begin{ENZam}\label{ENKerovGoodman}
For a given natural number $r$, let an infinite sequence $\tau=\{w_m\}_{m=1}^\infty\in({}^r\mathbb{YF})^\infty$ converge to a ``vertex at infinity'' $w\in{}^r\mathbb{YF}_\infty;$ also, let $\beta\in(0,1],$ $k\in\mathbb{N}.$ Then
\begin{gather*}
  \lim_{m\to\infty}\frac{\pi(w_m)}{\pi(w)}=\beta  
\end{gather*}
if and only if
\begin{gather*}
  \lim_{m\to\infty}\frac{\pi_k(w_m)}{\pi_k(w)}=\beta^k.
\end{gather*}    
\end{ENZam}

Thus, for a given natural number $r$, consider a sequence of vertices $\tau'=\{w'_m\}_{m=1}^\infty\in({}^r\mathbb{YF})^\infty,$ such that $|w'_m|\to\infty.$ If, for any vertex $v\in{}^{r}\mathbb{YF}$, the limit
\begin{gather}\label{ENlim10}
  \mu_{\tau'}(v):=\lim_{m\to \infty} d(\varepsilon,v)\frac{d(v,w'_m)}{d(\varepsilon,w'_m)},  
\end{gather}
exists, then the sequence $\tau'$ induces a measure $\mu_{\tau'}$ on the space of infinite paths in the graph ${}^{r}\mathbb{YF},$ such that the measure of the set of all paths passing through the vertex $v,$ equals $\mu_{\tau'}(v),$ and all $d(\varepsilon,v)$ initial segments of paths from the root to $v$ have the same measure $\mu_{\tau'}(v)/d(\varepsilon,v).$ These measures form the so-called \emph{Martin boundary} of the graph ${}^r\mathbb{YF}.$ Also, since the expression under the limit on the right-hand side of \eqref{ENlim10} is the ratio of the number of paths in $T(\varepsilon,w_m'),$ passing through the vertex $v,$ to the total number of paths in $T(\varepsilon,w'_m),$ it clearly follows that all these measures are probability measures in the sense that, for any $n\in\mathbb{N}_0$, we have
\begin{gather*}
    \sum_{v\in{}^r\mathbb{YF}_n} \mu_{\tau'}(v)=1.
\end{gather*} For $r=1$, such measures are described in \cite{ENGK}. Our task is to describe all such measures for $r\ge 2.$

Clearly, the sequence $\tau'$ has an infinite subsequence $\tau=\{w_m\}_{m=1}^\infty\in({}^r\mathbb{YF})^\infty,$ such that $\tau\to w\in{}^r\mathbb{YF}_\infty,$ and exactly one of the following two conditions holds:
\begin{enumerate}
    \item \label{ENu1} $w\in{}^r\mathbb{YF}_\infty^+,$ and there exists $\beta\in(0,1],$ such that $\pi(w_m)/\pi(w)\to\beta;$
    \item \label{ENu2} $\pi(w_m)\to 0.$
\end{enumerate}
Clearly, to describe the Martin boundary, it suffices to consider only such sequences $\tau$.

\renewcommand{\labelenumi}{\arabic{enumi}$)$}

The following Lemma \ref{ENtgk} and Theorem \ref{ENtop} generalize the result of \cite{ENGK}:
\begin{ENLemma}\label{ENtgk} 
Let $r\in\mathbb{N}.$ Then, for any infinite sequence of vertices $\tau=\{w_m\}_{m=1}^\infty\in({}^r\mathbb{YF})^\infty,$ such that $\tau\to w\in{}^r\mathbb{YF}_\infty$ and one of the conditions \textup{\ref{ENu1}),\;\ref{ENu2})} holds, the limit $\mu_\tau(v)$ exists for every vertex $v\in{}^r\mathbb{YF},$ and
\begin{enumerate}
    \item If the sequence $\tau$ satisfies condition \textup{\ref{ENu1})}, then the value $\mu_\tau(v)$ depends only on $v,$ $w,$ $\beta.$ Moreover, for any $w\in {}^r\mathbb{YF}_\infty^+,$ $\beta\in(0,1]$, there exists a sequence $\tau,$ corresponding to these $w,\beta.$
    \item If the sequence $\tau$ satisfies condition \textup{\ref{ENu2})}, then
    \begin{gather*}
        \mu_\tau(v)=\frac{d(\overline{\varepsilon},\overline{v})^2}{|v|!\cdot r^{e(v)}},
    \end{gather*}
that is, the value $\mu_\tau(v)$ depends only on $v.$
\end{enumerate}
\end{ENLemma}
\begin{proof}
The lemma in the case $r=1$ was proved in \cite[Theorem~8.7]{ENGK}. However, our proof does not rely on this and also applies to the case $r=1.$

Fix a vertex $v\in{}^r\mathbb{YF}.$ Clearly, we may assume that, for all $m$, we have $h'(w_m,w)\ge |v|.$ Thus, by \eqref{ENp2} and Theorem \ref{ENpaths}, and the fact that $e(v,w_m)=e(v,w),$ $\mathbbm{1}_{v,w_m}=\mathbbm{1}_{v,w},$ the expression
\begin{gather}\label{ENlimit}
    \lim_{m\to \infty} d(\varepsilon,v)\cdot\frac{d(v,w_m)}{d(\varepsilon,w_m)}
\end{gather}
equals the limit as $m$ tends to infinity of the product of the expression
\begin{gather}\label{ENmnozh}
d(\overline{\varepsilon},\overline{v})\cdot r^{d(v)}\cdot\frac{ r^{s(v,w_m)}}{ r^{d(w_m)}}=\frac{d(\overline{\varepsilon},\overline{v})}{r^{e(v)}}.
\end{gather}
and a weighted sum, over $k\in\overline{e(v,w)}$ or over $k\in\overline{e(v,w)}+1$, of the expressions
\begin{gather}\label{ENslag}
    \frac{d(\overline{v}\{k\},\overline{w_m}\{k\})}{d(\overline{\varepsilon},\overline{w_m})},
\end{gather}
where, for $k=0$, the coefficient of the corresponding summand is one.

The expression \eqref{ENmnozh} depends only on $v,$ so we consider the limit of the expression \eqref{ENslag} for a fixed $k\in\overline{e(v,w)}$ or $k\in\overline{e(v,w)}+1.$ By Theorem \ref{ENevtuh}, it equals
\begin{gather*}
\lim_{m\to\infty}\frac{\sum_{i=0}^{|{v\{k\}|}}\left({f(\overline{v}\{k\},i,h(\overline{v}\{k\},\overline{w_m}\{k\}))}\prod_{j=1}^{d(w_m\{k\})}(g(w_m\{k\},j)-i)\right)}{\prod_{j=1}^{d(w_m)}g(w_m,j)}.
\end{gather*}

Again, consider each summand in this expression. Note that $h(\overline{v}\{k\},\overline{w_m}\{k\})=h(\overline{v}\{k\},\overline{w}\{k\}),$ and also that, for a fixed $i\in\overline{|{v}\{k\}|}$,
\begin{align*}
\lim_{m\to\infty}\frac{\prod_{j=1}^{d(w_m\{k\})}(g(w_m\{k\},j)-i)}{ \prod_{j=1}^{d(w_m)}g(w_m,j)}=&\lim_{m\to\infty}\frac{\prod_{j=d(w_m\langle k\rangle)+1}^{d(w_m)}(g(w_m,j)-|w_m\langle k\rangle|-i)}{ \prod_{j=1}^{d(w_m)}g(w_m,j)}\\    
=&\frac{1}{\prod_{j=1}^{d(w\langle k\rangle)}g(w,j)}\cdot\lim_{m\to\infty}\prod_{j=d(w\langle k\rangle)+1}^{d(w_m)}\frac{g(w_m,j)-|w\langle k\rangle|-i}{g(w_m,j)}.
\end{align*}

Next, note that the following limiting relation holds:
\begin{gather*}
\lim_{m\to\infty}\prod_{j=d(w\langle k\rangle)+1}^{d(w_m)}\frac{g(w_m,j)-|w\langle k\rangle|-i}{g(w_m,j)}=\beta^{|w\langle k\rangle|+i}\cdot\prod_{j=d(w\langle k\rangle)+1}^{d(w)}\frac{g(w,j)-|w\langle k\rangle|-i}{g(w,j)}.
\end{gather*}
Indeed, it follows from Remark \ref{ENKerovGoodman}: the factors for which $g(w_m,j)\le |w\langle k\rangle|+i$ stabilize and coincide with the factors on the right for which $g(w,j)\le |w\langle k\rangle|+i$, while the remaining factors form the products $\pi_{|w\langle k\rangle|+i}(w_m)$ on the left and $\beta^{|w\langle k\rangle|+i} \pi_{|w\langle k\rangle|+i}(w)$ on the right.

Thus, we have proved that, for each such summand, the limit as $m$ tends to infinity exists and its value depends only on $w,$ $\beta,$ $v.$

To justify that, for any $w\in {}^r\mathbb{YF}_\infty^+,$ $\beta\in(0,1]$, there exists a sequence $\tau,$ corresponding to these $w,\beta,$ one can refer to the fact that, for arbitrary $r$, this clearly follows from the same fact in the case $r=1,$ proved in \cite{ENGK}; alternatively, one can prove it directly using Remark \ref{ENKerovGoodman}, in exactly the same way as Corollary 3 in \cite{ENBE}.

The first assertion is proved.

To prove the second assertion, it is easy to see that, when $\pi(w_m)\to 0$, all summands vanish except the summand corresponding to $k=i=0,$ which, by \eqref{ENochev}, implies that the expression \eqref{ENlimit} in this case equals
\begin{gather*}
    \frac{d(\overline{\varepsilon},\overline{v})}{r^{e(v)}}\cdot f(\overline{v},0,h(\overline{v},\overline{w})) = \frac{d(\overline{\varepsilon},\overline{v})^2}{|{v}|!\cdot r^{e(v)}}.    
\end{gather*}
The lemma is proved.
\end{proof}

In what follows, denote the measures from the first assertion of Lemma \ref{ENtgk} by $\mu_{r,w,\beta},$ and those from the second assertion by $\mu_{r,P}.$ Also put $\mu_{w,\beta}:=\mu_{1,w,\beta},$ $\mu_{P}:=\mu_{1,P}.$ The latter measure is called the \emph{Plancherel measure} of the Young--Fibonacci graph. Explicit values of the measures $\mu_{w,\beta}(v)$ were found in \cite{ENBE}, and explicit values of the measures $\mu_{r,w,\beta}(v)$ can be computed from the proof of Lemma \ref{ENtgk}, but are omitted here because of their cumbersome form.

All of the above yields
\begin{ENtheorem}\label{ENtop}
    For every natural number $r$, the Martin boundary of the graph ${}^r\mathbb{YF}$ consists of the measure $\mu_{r,P}$ and measures of the form $\mu_{r,w,\beta}$ for all $w\in {}^r\mathbb{YF}_\infty^+,$ $\beta\in (0,1].$
\end{ENtheorem}

The distinctness of these measures for $r=1$ was proved in \cite[Theorem~8.7]{ENGK}, and their ergodicity in \cite[Theorem~4]{ENBE}. The distinctness of these measures in the general case has not yet been proved, but follows from their ergodicity, which will be proved in the next section.

\renewcommand{\labelenumi}{\arabic{enumi}$^\circ$}

Thus, for given $r\in\mathbb{N},$ $w\in {}^r\mathbb{YF}_\infty^+,$ $\beta\in (0,1]$ and a nonnegative integer $n$, the functions $\mu_{r,w,\beta}$ define probability measures on the sets ${}^r\mathbb{YF}_n.$ Expressions such as $\mu_{r,w,\beta}(A)$ are understood in this sense for subsets $A\subseteq{}^r\mathbb{YF}_n.$

The following proposition will play an important role in proving the ergodicity of the Martin boundary of the graphs ${}^r\mathbb{YF}.$
\begin{ENProp}\label{ENsuum}
Let $r\in\mathbb{N},$ $w\in{}^r\mathbb{YF}_\infty^+,$ $\beta\in(0,1].$ Then, for any set $A\subseteq{}^1\mathbb{YF}_n$ with $n\in\mathbb{N}_0$,
\begin{gather*}
\sum_{v:\; \overline{v}\in A} \mu_{r,w,\beta}(v) = \mu_{\overline{w},\beta}(A).  
\end{gather*}
\end{ENProp}
\begin{proof}
Clearly, if a sequence of vertices $\{w_m\}_{m=1}^\infty\in({}^r\mathbb{YF})^\infty$ defines the measure $\mu_{r,w,\beta},$ then the sequence $\{\overline{w_m}\}_{m=1}^\infty\in(\mathbb{YF})^\infty$ defines the measure $\mu_{\overline{w},\beta}.$ It is also clear that it suffices to prove the proposition only for singleton sets $A.$ For such sets, the proposition follows from \eqref{ENsumma1} and \eqref{ENp2}.
\end{proof}

\section{Ergodicity of the Martin boundary}
Introduce some additional notation:

Let $r\in\mathbb{N},$ $w\in{}^r\mathbb{YF}_\infty,$ $\beta\in(0,1],$ $n,\ell\in\mathbb{N}_0,$ $\varepsilon>0,$ $k\in\mathbb{N}.$ Put
\begin{align*}
R(w,\beta,n,\varepsilon):=&\{v\in{}^r\mathbb{YF}_n:\quad \pi(v)\in(\pi(w)(\beta-\varepsilon),\pi(w)(\beta+\varepsilon)) \};\\
    \overline{R}(w,\beta,n,\varepsilon):=&\{v\in{}^r\mathbb{YF}_n:\quad \pi(v)\notin(\pi(w)(\beta-\varepsilon),\pi(w)(\beta+\varepsilon)) \};\\
    {Q}(w,n,\ell):=&\{v\in{}^r\mathbb{YF}_n:\quad h'(v,w)\ge \ell\};\\
    \overline{Q}(w,n,\ell):=&\{v\in{}^r\mathbb{YF}_n:\quad h'(v,w)< \ell\};\\
    \widetilde{Q}(w,n,\ell):=&\{v\in{}^r\mathbb{YF}_n:\quad h'(\overline{v},\overline{w})< \ell\};\\
    \widehat{Q}(w,n,\ell,k):=&\{v\in{}^r\mathbb{YF}_n:\quad h'(\overline{v},\overline{w})\ge \ell,\;e(v,w)=k-1 \}.
\end{align*}

Also, let $r\in\mathbb{N},$ $w\in{}^r\mathbb{YF}\cup{}^r\mathbb{YF}_\infty.$ Then, for each $k\in\overline{1,e(w)}$, define the vertex $w(k)\in{}^r\mathbb{YF}\cup{}^r\mathbb{YF}_\infty,$ obtained from $w$ by replacing the $k$th one from the right by a two.

We now generalize the result of \cite{ENBE}:
\begin{ENtheorem} \label{ENt3}
Let $r\in\mathbb{N},$ $w\in{}^r\mathbb{YF}_\infty^+,$ $\beta\in(0,1],$ and fix $\ell\in \mathbb{N}_0$ and $\varepsilon>0.$ Then the measure $\mu_{r,w,\beta}$, for large $n$, is concentrated on $R(w,\beta,n,\varepsilon)\cap Q(w,n,\ell):$
\begin{gather*}
  \lim_{n\to \infty} \mu_{r,w,\beta}(\overline{R}(w,\beta,n,\varepsilon)\cup \overline{Q}(w,n,\ell))=0.
\end{gather*}
\end{ENtheorem}
\begin{proof}
For $r=1$, this theorem was proved in \cite[Corollary~4 and Theorem~3]{ENBE}. We turn to the case $r\ge2.$

Let the longest suffix of the vertex $w$ among those of rank at most $\ell,$ contain $e$ ones. Note that
\begin{gather*}
    \overline{Q}(w,n,\ell)=\widetilde{Q}(w,n,\ell)\cup\bigcup_{k=1}^{e} \widehat{Q}(w,n,\ell,k).
\end{gather*}

From Proposition \ref{ENsuum} and the case $r=1$, it is clear that
\begin{gather}\label{ENlim1}
    \lim_{n\to\infty} \mu_{r,w,\beta}(\overline{R}(w,\beta,n,\varepsilon)\cup\widetilde{Q}(w,n,\ell)) =\lim_{n\to\infty}\mu_{\overline{w},\beta}(\overline{R}(\overline{w},\beta,n,\varepsilon)\cup \overline{Q}(\overline{w},n,\ell))=0.
\end{gather}

It remains to consider the sets $\widehat{Q}(w,n,\ell,k).$ Fix $k\in\overline{1,e}.$ Consider a sequence $\tau=\{w_m\}_{m=1}^\infty,$ inducing the measure $\mu_{r,w,\beta},$ such that, for all $m$, we have $h'(w_m,w)\ge \ell.$ In addition, consider the sequence of vertices $\{w_m(k)\}_{m=1}^\infty,$ converging to $w(k).$ Put
\begin{gather*}
    \beta(k):=\lim_{m\to \infty}\frac{\pi(w_m(k))}{\pi(w(k))}.
\end{gather*}

Let us estimate $\beta(k).$ Clearly, for all $m\in\mathbb{N}$,
\begin{gather*}
    \frac{\pi(w_m(k))}{\pi(w_m)}\ge \frac{1}{2},
\end{gather*}
which implies that $\beta(k)\in(0,1].$ It is also obvious that $\widehat{Q}(w,n,\ell,k)\subseteq \widetilde{Q}(w(k),n,\ell).$ Therefore, since \eqref{ENlim1} also holds for $\mu_{r,w(k),\beta(k)}(\widetilde{Q}(w(k),n,\ell)),$ to complete the proof of the theorem, it suffices to prove that there exists a constant $C,$ such that, for any $v\in \widehat{Q}(w,n,\ell,k),$ $m\in\mathbb{N}$,
\begin{gather*}
\frac{d(v,w_m)}{d(\varepsilon,w_m)}\le C\cdot\frac{d(v,w_m(k))}{d(\varepsilon,w_m(k))}.
\end{gather*}

First, note that, by \eqref{ENp2} and the definition of the function $g,$ for all natural numbers $m$,
\begin{gather*}
    \frac{d(\varepsilon,w_m(k))}{d(\varepsilon,w_m)}\le r\cdot \ell\cdot \prod_{j=1}^{d(w_m)}\frac{g(w_m,j)+1}{g(w_m,j)}\le \frac{2\cdot r\cdot \ell}{\pi(w_m)}\xrightarrow[m\to\infty]{} \frac{2\cdot r\cdot \ell}{
    \beta\cdot\pi(w)}.
\end{gather*}

Thus, to prove the theorem, it suffices to prove that, for any $v\in \widehat{Q}(w,n,\ell,k),$ $m\in\mathbb{N}$, we have $d(v,w_m)\le d(v,w_m(k)).$ Consider a path $t\in T(v,w_m).$ Note that the $k$th one from the right in the vertex $w_m$ could not have been initially present in $v.$ Associate injectively with this path a path $t'\in T(v,w_m(k))$ as follows: after the step of the path $t,$ at which the $k$th one from the right in the vertex $w_m$ is created, add an additional step at which this one immediately turns into a two, and leave all the other steps unchanged.

The theorem is proved.
\end{proof}

\begin{ENtheorem}
Let $r\in\mathbb{N}.$ Then every measure of the Martin boundary of the graph ${}^r\mathbb{YF}$ is ergodic, that is, it is not a mixture of other measures of this type.
\end{ENtheorem}
\begin{proof}
The ergodicity of $\mu_P$ was proved in \cite{ENKerGned}, and the proof for $\mu_{r,P}$ is quite analogous. The ergodicity of the measures $\mu_{w,\beta}$ was proved in \cite[Theorem~4]{ENBE}, and the proof for $\mu_{r,w,\beta}$ using Theorem \ref{ENt3} is entirely analogous.
\end{proof}

\section{Acknowledgments}
I am very grateful to my academic adviser, Fedor Vladimirovich Petrov, for proposing the problem, for his patience in checking the proof and editing the paper, and for his moral support throughout the entire period of work.

This work was supported by Saint Petersburg State University, project No. 159547043.

\section{AI disclosure}

The translation from Russian original was assisted by 
ChatGPT-6.1 Sol. Mathematical content is due to the author.

\end{document}